\documentclass[a4paper,11pt]{article}
\usepackage[cp1251]{inputenc}
\usepackage{amssymb,amsmath,amsthm}
\usepackage{graphics,graphicx,color}

\usepackage{srcltx, bm}
\DeclareGraphicsExtensions{.eps} \textwidth=169mm
\definecolor{darkred}{rgb}{0.9,0.1,0.1}

\newtheorem{theorem}{Theorem}[section]
\newtheorem{lemma}{Lemma}[section]

\newtheorem{remark}{Remark}[section]
\newtheorem{definition}{Definition}[section]{\rm}
\definecolor{darkred}{rgb}{0.9,0.1,0.1}

\definecolor{darkred}{rgb}{0.9,0.1,0.1}

\begin{document}

\title{Negative spectrum of non-local operators with periodic potential \thanks{ The study of
S. A. Pirogov was supported by the Russian Science Foundation (project 24-11-00123, https://rscf.ru/
en/project/24-11-00123/) }}

\author{S. Pirogov\thanks{
Institute for Information Transmission Problems (Kharkevich Institute) of the Russian Academy
of Sciences, Moscow, Russia, and National Research University Higher School of Economics, Moscow,
Russia  (s.a.pirogov@bk.ru) }
\and
\setcounter{footnote}{2}
E. Zhizhina\thanks{
Higher School of Modern Mathematics, Moscow Institute of Physics and Technology
(National Research University), Dolgoprudny, Russia  (elena.jijina@gmail.com)} }

\date{}
\maketitle

\begin{abstract}

The paper deals with spectral analysis of non-local operators arising in population dynamics models. We consider negative periodic perturbations of non-local operators of the convolution type. Such operators describe evolutions of the first correlation function in the stochastic birth and death dynamcis in the presence of suppression forces  that increase mortality. We consider the case when the birth kernel can be non-symmetric and spatially heterogeneous. It has been proven that any negative periodic perturbation of the equilibrium dynamics generator shifts the spectrum to the left half-plane and, consequently, such a perturbation of mortality leads to the population extinction in any dimension.
\\

\end{abstract}

\hskip 7 cm
\begin{minipage}[]{10 cm}
{\it   In memory of our friend Zhenia Pechersky}
\end{minipage}

\section{Introduction}

The paper deals with the spectral analysis of non local convolution type operators with a periodic non-positive potential of the form
\begin{equation}\label{1}
{\cal L} u(x) = \int\limits_{\mathbb{R}^d} a(x-y)(u(y) - u(x)) dy + V(x) u(x),
\end{equation}
where $a(\cdot) \ge 0, \; a(\cdot) \in L^1(\mathbb{R}^d)$, and the potential $V$ is a bounded periodic non-positive function, which is negative on a set of positive measure:  $V(\cdot) \le 0, \;  V \in L^\infty(\mathbb{R}^d), \;  V(x) = V(x+n), \; n \in \mathbb{Z}^d$.

This work is a natural continuation of our research at \cite{KMPZ}, where we proved the existence of a positive ground state for convolution operators perturbed by a local positive potential. Here, we are interested in the opposite case, when the potential $V$ is non-positive, and we want to answer the question of what negative potential is sufficient to shift the spectrum below 0.

The spectral properties of the convolution operators perturbed by a potential
have been investigated recently in \cite{KMPZ, KMV}, see also \cite{BPZ22, BPZ23a, BPZ23b, BPZ25}.
Such operators appear in the evolution equations for the first correlation function of the stochastic contact model in continuum, i.e. they describe the evolution of the density of particles in the corresponding birth and death infinite-particle system, see  e.g. \cite{FK}, \cite{KKP},  \cite{KPZ}. The contact model has a remarkable property: the equation for the first correlation function is closed and not coupled with the equations for higher order correlation functions. Then the evolution equation for the first correlation function reads
$$
\frac{\partial u(x,t)}{\partial t} = {\cal L} u(x,t), \quad u(x,0) = u_0(x).
$$
Thus, information about the structure of the spectrum of ${\cal L}$, the maximum eigenvalue, and the ground state is very important for predicting the growth or extinction of various populations whose evolution is determined by the mechanism of birth and death specified by the generator.

Unlike our previous works, where the potential was localized in space and the spectrum of operator ${\cal L}$ was studied in the space $L^2(\mathbb{R}^d)$, in the present paper we consider the potential $V(x)$ that is periodic with period 1.
This leads us to an operator acting in the space of periodic functions $u \in L^2(\mathbb{T}^d)$ and essentially changes the structure of the spectrum. In previous models, the essential spectrum was determined by both the convolution operator and the potential and contained 0, see e.g. Theorem 2.1 in \cite{BPZ23a} or Theorem 1 in \cite{BPZ25}.
In this model with a periodic potential, the structure of the essential spectrum changes radically. The essential spectrum is defined now as the range of the potential shifted away from 0,
and the convolution operator in $\mathbb{R}^d$, "wound" on a torus, becomes a compact operator in $L^2(\mathbb{T}^d)$.

The present work is a continuation of our previous study of the contact model in  \cite{PZ25}, where the existence of the stationary measure in the contact model with periodic birth and death rates has been proved. Here we are interested in the question: how strong must a negative external perturbation be to turn an infinite particle system from the stationary regime to the extinction?


\section{Statement of the problem}

In this paper, we consider several operators describing the evolution of the first correlation function in contact models with periodic birth and death rates in the presence of external perturbations.

We start with operator $L$ in  $L^2(\mathbb{T}^d)$ obtained from \eqref{1}:
\begin{equation}\label{L}
L u(x) = \int\limits_{\mathbb{T}^d} \tilde a(x-y)(u(y) - u(x)) dy + V(x) u(x), \quad u \in L^2(\mathbb{T}^d),
\end{equation}
where
\begin{equation}\label{tilde-a}
\tilde a(\cdot) \ge 0, \qquad \tilde a(z) = \sum\limits_{n \in \mathbb{Z}^d} a(z+n), \quad \tilde a(\cdot) \in L^1(\mathbb{T}^d).
\end{equation}
Notice that $\|a\|_{L^1(\mathbb{R}^d)} = \|\tilde a\|_{L^1(\mathbb{T}^d)}$, and in what follows we take $\|a\|_{L^1(\mathbb{R}^d)} = 1$.
\\

 Along with operator $L$, we also consider operator $M$, which has a more general form:
\begin{equation}\label{M}
M u(x) = \int\limits_{\mathbb{T}^d} b(x, y)(u(y) - u(x)) dy + V(x) u(x), \quad u \in L^2(\mathbb{T}^d).
\end{equation}
where
\begin{equation}\label{b-2}
 b(x, y) \ge 0, \quad b \not \equiv 0,  \qquad  0< \gamma_1 := \inf\limits_{x \in \mathbb{T}^d}  \int\limits_{\mathbb{T}^d}  b(x, y) \, dy  \le \sup\limits_{x \in \mathbb{T}^d} \int\limits_{\mathbb{T}^d}  b(x, y) \, dy =: \gamma_2 < \infty;
\end{equation}
\begin{equation}\label{b-2bis}
0< 
\int\limits_{\mathbb{T}^d}  b(x, y) \, dx \le \gamma_3 < \infty \quad \mbox{for all } \; y \in \mathbb{T}^d,
\end{equation}
and assume that a $n$-th iteration of the kernel $b(x,y)$ with some $n \in \mathbb{N}$ is strictly positive:
\begin{equation}\label{b-3}
\exists \ n \in \mathbb{N} \ \mbox{and } \ \exists \beta>0 \quad \mbox{such that} \  \int\limits_{(\mathbb{T}^d)^n}  b(x, y_1) \ldots b (y_n, y) \, dy_1 \ldots dy_n \ge \beta.
\end{equation}
 We also suppose that uniformly in $x$ there is integrability of $b(x,y)$ with respect to $y$, i.e.
\begin{equation}\label{b-compact}
\forall \ \varepsilon>0 \ \exists \ \delta>0 \quad \mbox{such that } \quad \sup\limits_{x} \int\limits_{e}  b(x, y) \, dy < \varepsilon \ \  \mbox{for all } \ e \subset \mathbb{T}^d, \ \mu(e)<\delta.
\end{equation}
\\
{\bf Example.} Bounded kernels $b(x,y) \in L^\infty(\mathbb{T}^{2d})$ or  kernels of the form $b(x,y)=q(x,y) \tilde a(x-y)$ with $\tilde a \in L^1(\mathbb{T}^d), \ q \in L^\infty(\mathbb{T}^{2d})$ satisfy all the above conditions.

\begin{remark}
It is worth noting that we do not assume that the functions $a(x-y)$ and $b(x,y)$ are symmetric, that is, the operators $L$ and $M$ may not be self-adjoint.
\end{remark}

\begin{remark}
The operator $M$ is the result of a "winding" on the torus of a non local convolution type operators with a periodic potential of the form
\begin{equation}\label{1bis}
{\cal M} u(x) = \int\limits_{\mathbb{R}^d} b(x,x-y)(u(y) - u(x)) dy + V(x) u(x),
\end{equation}
where $b(x,z) \ge 0, \; b(x,\cdot) \in L^1(\mathbb{R}^d), \; b(x,z)= b(x+n,z), \; n \in \mathbb{Z}^d $ and the potential $ V \in L^\infty(\mathbb{R}^d), \;  V(x) = V(x+n), \; n \in \mathbb{Z}^d$ is bounded, periodic and $V<0$ on a set of a positive measure.
Similar operators arise when studying stationary regimes of continuous contact models in a periodic medium, see \cite{PZ25}.
\end{remark}

\begin{definition}
An eigenvalue of a bounded linear operator is said to be a maximum eigenvalue, if it has a maximal real part.
\end{definition}


The main result of the present paper is the following theorem.

\begin{theorem}\label{T1}

Consider the operator $M$ acting in $ L^2(\mathbb{T}^d)$ by formula \eqref{M}. We assume that the kernel $b(x,y)$ satisfies all the conditions \eqref{b-2} - \eqref{b-compact}, and
\begin{equation}\label{V}
 V(\cdot) \le 0, \quad V \in L^\infty(\mathbb{T}^d), \quad V(\cdot)<0 \; \mbox{on a set of positive measure}.
\end{equation}

Then the spectrum of the operator $M$
lies in the negative half-space. The essential spectrum $\sigma_{ess}(M)$ of $M$ coincides with the essential range of the function $V(x) - W(x)$, where
$W(x) = \int\limits_{\mathbb{T}^d}  b(x, y) \, dy$:
\begin{equation}\label{essspec}
\sigma_{\rm ess}(M) = {\mathfrak{ Ran}}\, (V(x) -W(x)) \subset [-\alpha_0, -\alpha_1],
\end{equation}
where
\begin{equation}\label{essspec-b}
 \alpha_0 = {\rm ess} \,\sup (W(x)-V(x))\ge \gamma_2, \qquad \alpha_1 = {\rm ess} \, \inf (W(x)-V(x))\ge \gamma_1.
\end{equation}

The discrete spectrum $\sigma_{disc}(M)$ is not empty, $M$ has the maximum eigenvalue $\lambda, \ -\alpha_1 <\lambda<0 $, which is 
real and negative for any $V$ satisfying \eqref{V}.

\end{theorem}
For the proof of Theorem \ref{T1}, see section \ref{Th1-proof}.
\\

Thus, we conclude that any negative {\bf periodic} perturbation $V(x)$ of the equilibrium dynamics generator
$\int\limits_{\mathbb{R}^d} b(x,x-y)(u(y) - u(x)) dy$ with periodic in $x$ kernel $b(x,x-y)$
leads to the extinction of the evolving population in any dimension.
\medskip


In section \ref{S2}, we obtain estimates on the spectral gap, i.e. on the distance between 0 and the spectrum of the operator $L$. It will be shown that the gap depends on the potential $V$ and some characteristics of the kernel $\tilde a(\cdot)$, see Theorem \ref{T2}.

\begin{theorem}\label{T2}
Let $L$ be the operator given by \eqref{L} whose kernel $\tilde a(\cdot)$ is defined in \eqref{tilde-a}.
Then for any $ V(\cdot) \le 0, \;  V \in L^\infty(\mathbb{T}^d)$, such that $V(\cdot)$ is negative on a set of positive measure, the maximum eigenvalue $\lambda$ of $L$, 
which is negative by Theorem \ref{T1}, admits the following bound
\begin{equation}\label{t2}
\lambda< - \min \big\{ c_2 \gamma_0^2, \ \frac12 c_1  \big\},
\end{equation}
where $c_1 = \|V\|_{L^1(\mathbb{T}^d)}, \  \gamma_0 = \frac29 \, \frac{c_1}{\|V\|_{L^2(\mathbb{T}^d)}}$, and
$$-c_2 = \max\limits_{k \in \mathbb{Z}^d, \, k \neq 0} |a_k| \,- \, 1<0,$$
where $a_k$ are the Fourier coefficients of the function $\tilde a(\cdot)$.

\end{theorem}


\section{Proof of Theorem \ref{T1}}\label{Th1-proof}

Let us denote by $B$ the operator
\begin{equation}\label{B}
B u(x) = \int\limits_{\mathbb{T}^d} b(x,y) u(y) dy,  \quad u \in L^2(\mathbb{T}^d).
\end{equation}

\begin{lemma}
The operator $B$ defined by \eqref{B} is a compact operator in $L^2(\mathbb{T}^d)$.
\end{lemma}

\begin{proof}
First, note that $B$ is a bounded operator in $L^2(\mathbb{T}^d)$ according to the Schur test:
\begin{equation}\label{T1-0}
\|B\|_{L^2(\mathbb{T}^d) } \le  \sqrt{ \sup \limits_x \int\limits_{\mathbb{T}^d}  b(x, y) \, dy \cdot \sup \limits_y \int\limits_{\mathbb{T}^d}  b(x, y) \, dx }  \le \sqrt{\gamma_2 \, \gamma_3}.
\end{equation}

To prove the compactness of $B$ one can use the approximation $B_n$ of $B$ by the operators with truncated kernels
$$
b_n(x,y) =  \min \{ b(x,y), \ n  \}, \quad n \in \mathbb{N}.
$$
The kernels $b_n(x,y)$ are bounded for any $n$, and the operators $$B_n u(x) = \int\limits_{\mathbb{T}^d} b_n(x,y) u(y) dy$$ are compact operators in $L^2(\mathbb{T}^d)$.

Let us denote $$\theta_n(x) = \int\limits_{\mathbb{T}^d} (b(x,y) - b_n(x,y)) dy.$$  Then $\theta_n(x)$ is a bounded non-negative decreasing sequence, and $\theta_n(x) \to 0 $ by the Lebesgue's theorem. Moreover, from the Markov inequality and condition \eqref{b-2} it follows that for any $\delta>0$ and all $n> \frac{\gamma_2}{\delta}$
$$
\mu \big( y: \ b(x,y)>n \big) \ \le \ \frac{\gamma_2}{n}\ <  \ \delta,
$$
and then condition  \eqref{b-compact}  implies that for any $\varepsilon>0$ there exists $\delta>0$ such that
$$
\sup\limits_x \theta_n(x) < \sup\limits_x \int\limits_{y: \ b(x,y)>n} b(x,y) dy <  \varepsilon \quad \mbox{ for all } \quad  n > \frac{\gamma_2}{\delta},
$$
since $\mu ( y: \ b(x,y)>n )<\delta$. Consequently,
\begin{equation}\label{T1-1}
\theta_n(x) \to 0 \quad \mbox{uniformely in } \, x, \ \mbox{i.e. } \
\lim\limits_{n \to \infty} \sup\limits_x \theta_n(x) = 0.
\end{equation}
Now, using relation \eqref{T1-1} we will prove that
\begin{equation}\label{T1-1bis}
\|B-B_n \|_{L^2(\mathbb{T}^d) \to L^2(\mathbb{T}^d)} \to 0 \quad  \mbox{as } \quad n \to \infty.
\end{equation}
The Cauchy-Schwartz inequality yields:
\begin{equation}\label{T1-2}
\begin{array}{c}
\displaystyle
(B - B_n) f(x) = \int\limits_{\mathbb{T}^d} (b(x,y) - b_n(x,y)) f(y) dy
\\[2mm] \displaystyle
 \le \Big(  \int\limits_{\mathbb{T}^d} (b(x,y) - b_n(x,y)) dy \Big)^{1/2} \Big(  \int\limits_{\mathbb{T}^d} (b(x,y) - b_n(x,y)) f^2(y) dy \Big)^{1/2}.
 \end{array}
\end{equation}
Consequently,
\begin{equation}\label{T1-3}
\begin{array}{l}
\displaystyle
\int\limits_{\mathbb{T}^d}\big((B-B_n) f(x)\big)^2  dx \ \le \  \int\limits_{\mathbb{T}^d} \big[  \int\limits_{\mathbb{T}^d} (b(x,y) - b_n(x,y)) dy  \big] \big[  \int\limits_{\mathbb{T}^d} (b(x,y) - b_n(x,y))f^2(y) dy  \big] \, dx \\[2mm]
\displaystyle
= \int\limits_{\mathbb{T}^d} \theta_n(x)  \int\limits_{\mathbb{T}^d} (b(x,y) - b_n(x,y))f^2(y) dy \, dx
= \int\limits_{\mathbb{T}^d}  f^2(y)  \int\limits_{\mathbb{T}^d} \theta_n(x)  (b(x,y) - b_n(x,y)) dx \, dy
 \\[2mm]
\displaystyle
\le   \sup\limits_y  \int\limits_{\mathbb{T}^d}   \theta_n(x)  (b(x,y) - b_n(x,y))  dx \ \int\limits_{\mathbb{T}^d}  f^2(y) dy \le  \|f\|^2 \, \sup\limits_y  \int\limits_{\mathbb{T}^d}   \theta_n(x)  b(x,y)  dx.
\end{array}
\end{equation}
Here we used Fubini's theorem. Due to \eqref{b-2bis} and \eqref{T1-1} we get
\begin{equation}\label{T1-4}
 \sup\limits_y  \int\limits_{\mathbb{T}^d}   \theta_n(x)  b(x,y)  dx \le  \sup\limits_x  \theta_n(x) \ \sup\limits_y \int\limits_{\mathbb{T}^d} b(x,y) dx \to 0, \quad n \to \infty.
\end{equation}
Consequently, \eqref{T1-1bis} holds, and $B$ is a compact operator.
\end{proof}

Thus the operator $M$ given by \eqref{M} is written as follows:
\begin{equation}\label{M-1}
M u(x) = B u(x) + \big( V(x)-W(x) \big) \, u(x), \quad W(x) =  \int\limits_{\mathbb{T}^d} b(x,y) dy, \quad u \in L^2(\mathbb{T}^d),
\end{equation}
and it is the sum of the operator of multiplication by the real-valued function $V-W$ and the compact operator $B$. Consequently, $\sigma(M) = \sigma_{\rm ess}(M) \cup \sigma_{\rm disc}(M) $, where the essential spectrum $\sigma_{\rm ess}(M)$ coincides with the essential range of the function $(V(x) -W(x))$,
and the discrete spectrum lies in $\mathbb{C} \setminus {\mathfrak{ Ran}}\, (V(x) -W(x))$. The essential range of the function $(V(x) -W(x))$ lies in the negative half-plane: ${\mathfrak{ Ran}}\, (V(x) -W(x)) \subset [-\alpha_0, \, - \alpha_1]$, where $\alpha_0 \ge \gamma_2, \  \alpha_1 \ge \gamma_1$.
\medskip

Next we will prove that the discrete spectrum of the operator $M$ lies in the negative half-plane.
\medskip


By adding to the operator $M$ a positive operator $kI$ with a proper positive constant $k>0$, we obtain a positive operator $T= M + kI$.
The structure of the spectra of the operators $M$ and $T$ is the similar, the spectra differ only in a shift to the right by $k$.
Denote by $\sigma(T)$ and $\sigma_{ess}(T)$ the spectrum and the essential spectrum of $T$ respectively, and let $r_e(T) = \sup \{|\mu|: \ \mu \in \sigma_{ess}(T) \}$.

Next we will need the following statement.

\begin{theorem}[Theorem 1, \cite{S}]\label{TS}
Let $K$ be a reproducing cone in a real Banach space $X$, and $T: X \to X$
be a bounded, positive linear operator. Suppose there exists $\mu \in \sigma (T)$ such that
$|\mu | > r_e(T)$. Then $T$ has a positive eigenvalue $\rho$ such that $\rho \ge |\lambda|$ for all $\lambda \in \sigma(T)$ ,
and to $\rho$ there corresponds at least one eigenvector $v \in K$ of $T$ and at least one eigenvector
$v^* \in K^*$ of the adjoint operator $T^*$.
\end{theorem}

This theorem says that if the operator $T = M+kI$ has an eigenvalue $\rho$ of maximum modulus, such that $|\rho|> r_e(T) = -\alpha_1+k$, then this eigenvalue is positive (and, in particular, this eigenvalue is real). Consequently, the analogous statement holds for the operator $M$: if the operator $M$ has the eigenvalue $\lambda$ with the maximal real part such that ${\rm Re} \, \lambda> -\alpha_1$, then $\lambda$ is real,
and moreover,  $\lambda$ is also the maximum eigenvalue for the adjoint operator $M^*$.

Let us study when such eigenvalue $ - \alpha_1 < \lambda < 0$ exists. The equations on the maximum eigenvalue $\lambda$ and the corresponding leading eigenfunction $\psi$ and $\phi$ of the operator $M$ and $M^*$ read:
\begin{equation}\label{L2-bis}
M \psi = (B + V - W)\psi = \lambda \psi, \qquad
M^* \phi = (B^* + V - W)\phi = \lambda \phi.
\end{equation}
These equations can be rewritten as follows:
\begin{equation}\label{L2-2}
\begin{array}{c}
\displaystyle
Q_\lambda \psi (x) := \frac{1}{U(x)+W(x)+\lambda} \, \int\limits_{\mathbb{T}^d} b(x,y) \psi (y) dy = \psi (x), \\[2mm]
\displaystyle
Q^*_\lambda \phi (x) := \frac{1}{U(x)+W(x)+\lambda} \, \int\limits_{\mathbb{T}^d} b(y,x) \phi(y) dy = \phi (x),
\end{array}
\end{equation}
where $U(x)= -V(x) \ge 0, \ U \not \equiv 0$.

Let us consider the second equation in \eqref{L2-2}. Since $B^*$ is a compact operator, the operator $Q^*_\mu$ is compact and positive for any $\mu> - \alpha_1$.
We study the behavior of the spectral radius $r(Q^*_\mu)$ of the operator $Q^*_\mu$ as a function of $\mu$ when $\mu>-\alpha_1$. We use arguments similar to those given in the paper \cite{KMPZ}.

First we prove that $r( Q^*_{\mu})<1$ for $\mu \ge 0$. Indeed, if it is not true and  $r( Q^*_{\mu}) \ge 1$, then by the Krein-Rutman theorem for compact positive operators (\cite{KR}, Theorem 6) the following equality should be valid:
\begin{equation}\label{L2-3}
\frac{1}{U(x) + W(x) + \mu} \, \int\limits_{\mathbb{T}^d} b(y,x) \phi(y) dy = r \, \phi(x), \quad \mbox{ with } \; r = r( Q^*_{\mu}) \ge 1, \quad \mbox{for some} \quad \phi(x) >0.
\end{equation}
This equality can be rewritten as
\begin{equation}\label{L2-4}
\int\limits_{\mathbb{T}^d} b(y,x) \phi (y) dy = r (U(x)+W(x)+\mu) \phi(x), \quad  r \ge 1, \quad \mu \ge 0,  \quad \phi(x) >0.
\end{equation}
After integrating \eqref{L2-4} over $x \in \mathbb{T}^d$ 
we obtain
\begin{equation}\label{L2-5}
\int\limits_{\mathbb{T}^d} \int\limits_{\mathbb{T}^d} b(y,x) \phi (y) \, dy \, dx = r \int\limits_{\mathbb{T}^d} (U(x)+W(x)+\mu) \phi(x) dx.
\end{equation}
Since for the expression on the right hand side of the equality \eqref{L2-5} we have the estimate
\begin{equation}\label{L2-5bis}
\begin{array}{c}
\displaystyle
r \int\limits_{\mathbb{T}^d} (U(x)+W(x)+\mu) \phi(x) dx \ge r \int\limits_{\mathbb{T}^d} \int\limits_{\mathbb{T}^d} b(x,y) \, dy \,  \phi(x) \, dx + r \int\limits_{\mathbb{T}^d} U(x) \phi(x) \, dx
\\[2mm] \displaystyle
 >   r \int\limits_{\mathbb{T}^d} \int\limits_{\mathbb{T}^d} b(x,y)   \phi(x) \, dy \, dx =   r \int\limits_{\mathbb{T}^d} \int\limits_{\mathbb{T}^d} b(y,x)   \phi(y) \, dy \, dx,
\end{array}
\end{equation}
we conclude that the equality \eqref{L2-5} cannot be true for  $\mu \ge 0, \; r \ge 1, \; v(x) >0$ and $U(x) \ge 0, \; U(x) \not \equiv 0$.

Thus, $r( Q^*_{\mu})<1$ for $\mu \ge 0$, and the same holds for  $r( Q_{\mu})$. This implies, in particular, that $\lambda$ in equation \eqref{L2-bis} should be only negative: $\lambda<0$. Therefore, we proved that the spectrum of the operator $M$ lies in the negative half-plane.
\medskip

In the next part of this section we will prove that the discrete spectrum lying above the edge $-\alpha_1$ of the essential spectrum is not empty. 
We will need the following lemmas about the properties of the spectral radius of the compact positive operator $Q_\mu$.

\begin{lemma} \label{Lemma 2}
For $\mu > -\alpha_1$, the spectral radius $r( Q_{\mu})$ is continuous and
monotonically decreasing with respect to $\mu$. Moreover, $r(Q_\mu) \to 0$ for $\mu \to +\infty$.
\end{lemma}
\begin{proof}
The spectral radius of the compact positive operator $ Q_{\mu}$ coincides with its maximum eigenvalue $\Lambda(\mu)$, which is a simple isolated eigenvalue. Thus, applying the perturbation theory we conclude that the spectral radius $r( Q_{\mu})$ is continuous in $\mu$.

To prove the monotonic decrease of the spectral radius we use the fact that if $A$ and $B$ are positive operators, such that $A \le B$ in the order sense, then $r(A) \le r(B)$. Really, for such operators we have $\|A\| \le \|B\|$ and $\|A^n\| \le \|B^n\|$ for any $n \in \mathbb{N}$. Consequently, $r(A) \le r(B)$ due to formula for the spectral radius: $r(A) = \lim\limits_{n\to \infty} \sqrt[n]{\|A^n\|}$.
The positive kernels of the operators $ Q_{\mu}$
$$
Q_\mu (x,y) = \frac{b(x,y)}{U(x)+W(x)+\mu}
$$
is monotonically decreasing function of $\mu > -\alpha_1$. Consequently,
$$
Q_\mu \le Q_{\mu'} \quad \mbox{ when } \quad \mu \ge \mu'.
$$
Thus, the spectral radius $r( Q_{\mu})$ decreases monotonically with respect to $\mu$.


Additionally, the structure \eqref{L2-2} of the operators $Q_\mu$ implies that  $r(Q_\mu) \to 0$ as $\mu \to +\infty$.
\end{proof}

\begin{lemma}\label{Lemma3}
We have
\begin{equation}\label{L3-1}
\lim\limits_{\mu \to -\alpha_1+0} r(Q_\mu) >1, \quad \mbox{ where } \quad  \alpha_1 = {\rm ess}\, \inf(W(x)+U(x)).
\end{equation}
\end{lemma}
\begin{proof}

We apply the following useful statement about spectral radius of a positive operator.

\begin{lemma}[\cite{KR}, Theorem 6.2]
If $Q$ is a positive operator and if there exists a function $\varphi(x) \in L^2(\mathbb{T}^d), \ \varphi \ge 0, \ \| \varphi \|=1 $, such that
$$
Q \varphi(x) \ge c_0 \varphi(x),
$$
then
$$
r(Q) \ge c_0.
$$
\end{lemma}
\medskip

Let us take $\delta$ such that $\delta<\frac{\beta}{(\alpha_0 - \alpha_1)^{n-1}}$, where $\beta$ and $n$ were defined in condition \eqref{b-3}, and the constants $\alpha_0>\alpha_1>0$ were introduced in \eqref{essspec-b}. Denote by $B_\delta \subset \mathbb{T}^d$ a set of a positive measure such that
$$
W(x)+U(x)\le \alpha_1+\delta \quad \mbox{ for all } \quad x \in B_\delta.
$$
Let us consider function
$$
\varphi_\delta (x) = \frac{1}{\sqrt{|B_\delta|}} \chi_{B_\delta}(x),
$$
where $ \chi_{B_\delta}(\cdot)$ is a characteristic function of $B_\delta$. Then using condition \eqref{b-3} we obtain
\begin{equation}\label{L3-2}
Q^n_\mu \varphi_\delta (x) = \int Q_\mu^n (x,y) \varphi_\delta (y) dy \ge \frac{\beta}{(\alpha_1+\delta + \mu) (\alpha_0 + \mu)^{n-1}} \frac{ \chi_{B_\delta}(x)}  {\sqrt{|B_\delta|}} .
\end{equation}
Since
$$
\frac{\beta}{(\alpha_1+\delta + \mu) (\alpha_0 + \mu)^{n-1}} \ \to \ \frac{\beta}{\delta \, (\alpha_0 - \alpha_1)^{n-1}} \quad \mbox{as }\quad \mu \to -\alpha_1+0,
$$
it follows from inequality $ \frac{\beta}{\delta \, (\alpha_0 - \alpha_1)^{n-1}}>1$ that
\begin{equation}\label{L3-3}
\lim\limits_{\mu \to -\alpha_1+0} r(Q^n_\mu) >1,
\end{equation}
and consequently \eqref{L3-1} also holds.
\end{proof}

Using the monotonicity of the spectral radius with respect to $\mu$ together with relation \eqref{L3-1} we conclude, that there exists $\lambda, \ -\alpha_1< \lambda<0$, such that  $r(Q_{\lambda}) =1$.
Using the Krein-Rutman theorem again together with \eqref{L2-2} we obtain that there exist $\psi>0$ and $\phi>0$ such that the equations in \eqref{L2-bis} holds with this $\lambda$. Thus, in this case
the operator $M$ has at least one point of the discrete spectrum above the essential spectrum.
Moreover, by the construction, $\lambda $ for which $r(Q_{\lambda}) =1$ is the maximal eigenvalue of the operator $M$.
Consequently, the corresponding function $\psi(x) = \psi_\lambda(x)$ is the ground state of the operator $M$.

The uniqueness of $\psi_\lambda(x)$ in $L^2(\mathbb{T}^d)$ follows from the positivity improving property of the semigroups $e^{tM}$, see e.g. \cite{RS4}, Theorem XIII.44. This semigroup is positivity improving due to the property \eqref{b-3}.

Theorem \ref{T1} is completely proved.

\section{Proof of Theorem \ref{T2}. Estimation of a gap between 0 and $\lambda<0$}\label{S2}

Our next task is to estimate the distance from 0 to the spectrum of the operator $L$, or equivalently, the maximum eigenvalue $\lambda$ of $L$, which is real and negative according to Theorem \ref{T1}.

I. Let us assume first that $a(-z) = a(z)$ is a symmetric function.  We rewrite operator $L$ as $L=P+V$. where we denoted $P = A-E$, and operator $A$ is the convolution operator: $Au = \tilde a \ast u$, $\|\tilde a\|_{L^1(\mathbb{T}^d)} =1$.
We have
\begin{equation}\label{A-1}
(Pf,f) \le 0, \ (Vf, f) \le 0, \ P1=0, \ (V1,1) =\int_{T^d} V(\xi) d \xi = - \|V\|_1 =  - c_1,
\end{equation}
where $\| \cdot \|_1 = \|\cdot \|_{L^1(\mathbb{T}^d)}$.
Let us remind that $$\sigma_{{\rm ess}}(L) = { \mathfrak{Ran}} (V - 1) = [-\nu, -1], \quad -\nu = {\rm ess} \inf V -1<-1.$$

Since $A$ is a compact operator in $L^2(\mathbb{T}^d)$, then the operator $P=A-E$ has a discrete spectrum, and the maximum eigenvalue is equal to 0. Moreover, by the Krein-Rutman theorem, $1 \in L^2(\mathbb{T}^d)$ is the unique eigenfunction of $P$ corresponding to 0, since there exists a $n$-th iteration of the kernel $\tilde a(\cdot)$ which is strictly positive on $\mathbb{T}^d$, see e.g. Lemma 4.2 in \cite{PZh19}.

Consequently, $L^2(\mathbb{T}^d)$ is decomposed into a direct sum
\begin{equation}\label{sum}
L^2(\mathbb{T}^d) =  \{ 1 \} \, \oplus \, {\cal H}^\bot
\end{equation}
of subspaces invariant with respect to $P=A-E$, and since the maximal accumulation point of the discrete spectrum is $-1$, we get that
\begin{equation}\label{A-2}
(Pg,g) \le -c_2 (g,g), \quad \forall \; g \in {\cal H}^\bot.
\end{equation}
for some $0<c_2 \le 1$.

In what follows we will denote $\| \cdot \| = \| \cdot \|_{L^2(\mathbb{T}^d)} $. We will prove that there exists $\kappa>0$, such that $((P+V)f, f) \le - \kappa<0$ for all $f \in L_2(T^d), \ \|f\|=1$. To do this, we write  $f$ with $\|f \|=1$ as follows:
\begin{equation}\label{f-1}
f= f_1+f_2, \quad f_1 = \alpha \, 1, \ f_2 = \gamma \tilde f_2, \quad \tilde f_2 \in {\cal H}^\bot, \ \|\tilde f_2\|=1, \ \|f_1\| = \alpha, \ \|f_2\| = \gamma, \  \alpha^2 + \gamma^2 =1.
\end{equation}
It follows from \eqref{A-1} - \eqref{A-2} that
\begin{equation}\label{AV-2}
((P + V)f,f) = (P(f_1 + f_2), f_1+ f_2)+(Vf,f) \le (P f_2, f_2) \le - c_2 \gamma^2.
\end{equation}
We will apply estimate \eqref{AV-2} below with  $\gamma \ge \gamma_0>0$, where   $\gamma_0 \le \frac14$ will be chosen later.

Taking into account $(Pf,f) \le 0$ we have
  \begin{equation}\label{AV-3}
  \begin{array}{c}
((P+V)f,f) \le (V(f_1 + f_2), f_1+ f_2) \le (V f_1, f_1) + 2(V f_1, f_2) \\[2mm]
 \le - c_1 \alpha^2 + 2  \|V\| \alpha  \gamma = -c_1 (1- \gamma^2) + 2 \|V\| \gamma \sqrt{1- \gamma^2} \\[2mm] = - c_1 + \gamma^2 c_1 + 2 \| V \| \gamma \sqrt{1 - \gamma^2}
< - c_1 + \frac14 \gamma c_1 + 2 \| V\| \gamma (1 - \frac12 \gamma^2)\\[2mm]
 < - c_1 + \gamma \big( \frac14 c_1 +2 \|V\|  \big) < - \frac12 c_1,
\end{array}
\end{equation}
when
 \begin{equation}\label{gamma}
\gamma <  \frac{2 c_1}{c_1 + 8 \|V\|} =: \tilde \gamma_0.
\end{equation}
We used in \eqref{AV-3} the following evident estimates
$$
\sqrt{1- \gamma^2} < 1-\frac12 \gamma^2, \quad \gamma^2< \frac14 \gamma \quad \mbox{ for  } \; \gamma< \frac14.
$$
Since $c_1 = \|V\|_1 \le \|V\|$, then \eqref{gamma} yields 
 \begin{equation}\label{gamma-1}
\frac29 \, \frac{\|V\|_1}{\|V\|}  \le \tilde \gamma_0 \le \frac29.
\end{equation}
Therefore, taking $\gamma_0 = \frac29 \, \frac{\|V\|_1}{\|V\|}$ we get
   \begin{equation}\label{AV-4}
   \begin{array}{c}
   ((P+V)f,f) \le -c_2 \gamma_0^2 \|f\|^2 \ \mbox{ as } \ \|f_2\|\ge \gamma_0 \|f\| \\[3mm]
   ((P+V)f,f) \le -\frac12 c_1 \|f\|^2 \ \mbox{ as } \ \|f_2\|\le \gamma_0 \|f\|.
   \end{array}
   \end{equation}
 Thus,
   \begin{equation}\label{AV-5}
    ((P+V)f,f) \le - \min \big\{ c_2 \gamma_0^2, \ \frac12 c_1  \big\} \,  \|f\|^2
 \end{equation}
 and $\kappa = \min \big\{ c_2 \gamma_0^2, \ \frac12 c_1  \big\}$.
\\
\begin{remark}  The analogous bound can be obtained in the same way for the operator
$$
(L_\mu + V)u (x) = \int_{R^d} a(x-y)\mu(x,y) (u(y) - u(x)) dy + V(x)u(x), \quad u \in L^2(R^d),
$$
where $\mu(x,y) = \mu(y,x)$ is a positive symmetric periodic function satisfying the estimate  $0< \mu_- \le \mu(x,y) \le \mu_+ < \infty$.
\end{remark}

II. Now we do not assume  that $a(\cdot)$ is a symmetric function. Then the operator $L$ is not symmetric anymore, and the minimax principle is not applicable in this case.
However, the decomposition \eqref{sum} of $L^2(\mathbb{T}^d)$ into an orthogonal sum of subspaces invariant under $P=A-E$ remains valid for any convolution kernel $\tilde a(\cdot)$. Indeed,
$$
P1=0, \quad (Pg,1)= \int \int \tilde a(x-y) g(y) dy dx - (g,1) = 0 \; \mbox{ for all } \;  g \in   {\cal H}^\bot.
$$
Using the Fourier transform and the Parseval's identity we conclude that the estimate \eqref{A-2} holds with a constant $-c_2 = \max\limits_{k \in \mathbb{Z}^d, \, k \neq 0} |a_k| \,- \, 1<0$, where $a_k$ are the Fourier coefficients of the function $\tilde a(\cdot)$. Then, applying the same reasoning as above we conclude that ${\rm Re} (Lf,f)\le -\kappa \|f\|^2$ for any $f \in L^2(\mathbb{T}^d)$. Comparing this inequality with the equation on the leading eigenfunction $L \psi = \lambda \psi$ we get
\begin{equation}
(L\psi,\psi) = \lambda (\psi,\psi) \le -\kappa\|\psi\|^2.
\end{equation}
Since the maximal eigenvalue $\lambda$ of $L$ is real by Theorem \ref{T1}, we obtain that $\lambda \le -\kappa$, where $\kappa$ is defined as above in \eqref{AV-5}.



\begin{thebibliography}{20}


\bibitem{BPZ22} 
    D.I. Borisov, E.A. Zhizhina, A.L. Piatnitski, Spectrum of a convolution operator with potential, Russian Math. Surveys (2022), 77(3), 546-548.
    DOI 10.4213/rm10038.

\bibitem{BPZ23a} Denis I. Borisov, Andrey L. Piatnitski, Elena A. Zhizhina, On the spectrum of convolution operators with a potential, JMAA (Journal of mathematical analysis and applications), 2023, Vol. 517(1), 126568; https://doi.org/10.1016/j.jmaa.2022.126568

\bibitem{BPZ23b} D.I. Borisov, A.L. Piatnitski and E.A. Zhizhina, Spectrum of One-Dimensional Potential Perturbed by a Small Convolution Operator: General Structure, Mathematics (2023) 11(19), 4042  https://doi.org/10.3390/math11194042

\bibitem{BPZ25} D.I. Borisov, A.L. Piatnitski and E.A. Zhizhina, Convolution type operators with potential: essential and infinite discrete spectrum, MMAS (Mathematical Methods in the Applied Sciences) Vol. 48, Issue 8, pp. 8687-8695, (2025), http://doi.org/10.1002/mma.10745

\bibitem{FK} D. Finkelshtein, Yu. Kondratiev, O. Kutoviy, S. Molchanov, E. Zhizhina,
Density behavior of spatial birth-and-death stochastic evolution of mutating genotypes
under selection rates,
Russian Journal of Math. Physics, 2014, vol. 21, No 4, pp. 450-459.

\bibitem{S} D.E. Edmunds, A.J.B. Potter, C.A. Stuart, Non-compact positive operators, Proc. R. Soc. Lond. A, 328, 67-81 (1972)

\bibitem{KKP} Yu. Kondratiev, O. Kutoviy, S. Pirogov, Correlation functions
and invariant measures in continuous contact model, Ininite
Dimensional Analysis, Quantum Probability and Related Topics Vol.
11, No. 2, 231-258 (2008)


\bibitem{KMPZ} Yu. G. Kondratiev, S. Molchanov, S. Pirogov, E. Zhizhina, On ground state of some non local Schrodinger operator, Applicable Analysis, 96 (8), 2017,  pp. 1390-1400, http://dx.doi.org/10.1080/00036811.2016.1192138.

\bibitem{KMV} Yu. Kondratiev, S. Molchanov, B. Vainberg, Spectral analysis of non-local Schrodinger operators,
J. Funct. Anal., 273(3) 2017, pp. 1020–1048.

\bibitem{KPZ} Yu. Kondratiev, S. Pirogov, E. Zhizhina, A Quasispecies Continuous Contact Model
in a Critical Regime, Journal of Statistical Physics, 163(2), 357-373 (2016), doi:10.1007/s10955-016-1480-5


\bibitem{KR} Krein, M.G.; Rutman, M.A., Linear operators leaving
invariant a cone in a Banach space, Uspehi Matem. Nauk (in Russian)
3, p. 3-95 (1948). English translation: Krein, M.G.; Rutman, M.A.,
Linear operators leaving invariant a cone in a Banach space, Amer.
Math. Soc. Translation 26 (1950).


\bibitem{PZh19}  A. Piatnitski, E. Zhizhina, Homogenization of  biased convolution type operators,
Asymptotic Analysis, 2019, Vol. 115, No 3-4, pp. 241-262, doi:10.3233/ASY-191533.

\bibitem{PZ25}  S. Pirogov, E. Zhizhina, On existence of invariant measures of a continuous contact model in a periodic medium, SIAM Theory Probab. Appl., Vol.70, N 4, 2026.


\bibitem{RS4} M. Reed, B. Simon, Methods of modern mathematical physics, Vol.4, Academic Press, NY 1978



\end{thebibliography}
\end{document}